\documentclass[letterpaper,11pt]{article}
\usepackage[english]{babel}
\usepackage[ansinew]{inputenc}
\usepackage{fontenc}
\usepackage{amsfonts}
\usepackage{amsmath}
\usepackage{amsthm}
\usepackage{enumerate}
\usepackage{textcomp}
\usepackage{geometry}
\usepackage{mathrsfs}
\usepackage{natbib}
\usepackage[titletoc]{appendix}
\usepackage{pst-tree}

\usepackage[colorlinks=true,linkcolor=red,citecolor=blue]{hyperref}

\newcommand{\n}{\noindent}
\newtheorem{rem}{Remark}[section]
\newtheorem{prop}{Proposition}[section]

\begin{document}

\title{Performance analysis of Markovian queue with impatience and vacation times}
\author{Assia Boumahdaf \footnote{E-mail: assia.boumahdaf@etu.upmc.fr}\\
Laboratoire de Statistique Th\'eorique et Appliqu\'ee \\Universit\'e Pierre et Marie Curie, Paris, France}
\date{}
\maketitle

\begin{abstract}
We consider an $M/M/1$ queueing system with impatient customers with multiple and single vacations. It is assumed that customers are impatient whenever the state of the server. We derive the probability generating functions of the number of customers in the system and we obtain some performance measures.

\end{abstract}

\section{Introduction}
\label{Intro}

Queueing system with impatience are characterized by the fact that customers are not willing to wait as long as it is necessary to obtain service. Such models are encountered in many fields including real-time telecommunication systems in which data must be transmitted within a short time, call center in which customers hang up due to impatience before they are served.

Queueing models with impatient customers have been extensively studied in the past by various authors. The idea of impatient customers goes back to the pioneering work of Palm in $1917$. This notion was developed from the 50's with the work of \cite{Haight57} in 1957. 
The author considered a model of balking for $M/M/1$ queue in which there is a greatest queue length at which an arrival would not balk. 
\cite{AnckerI63} studied the $M/M/1$ queuing system with balking and reneging and performed its steady state analysis. The general model $GI/G/1 + GI$, has been thoroughly studied by \cite{Daley65}, \cite{Stanford79}, and  \cite{Baccelli84}. The model has recently studied by \cite{Moyal2010}, and \cite{Moyal2013}.
In all aforementioned papers, the server is always available to serve the customers in the system and the service process is never interrupted.

Vacation models are queueing systems in which servers may become unavailable for a period of time. In telecommunication systems, this period of absence may represent the server's working on some secondary job. In manufacturing systems, these unavailable periods may represent performing maintenance activities, or equipment breakdowns. These systems are received considerable attention in the literature, see for example, the book of \cite{TianZhang} and the survey of \cite{Doshi1986}.

Motivated by questions regarding the performance on real-time applications in packet switching network, this paper consider vacation models with impatient customers. Indeed, in a packet switching network, packets are routed from source to destination along a single path (arrival streams or packet flows) through intermediate nodes, called switch (queue). 
Packets (customers) from different sources are multiplexing resulting in delay transmission (waiting time). 
If we are concerned with performances of each packet flows, we may regarded the the other packet flows as secondary tasks. 
When the server is working on these secondary packet, it may be regarded as server on vacation. 
Moreover, a packet may loses  its value (renege) if it is not transmitted within a given interval. 
These both features can be modelled by a vacation queueing with customer impatience.

Queueing models with impatience and vacation seems to go back to \cite{vanderDuynSchouten78}. The author considers the model $M/G/1$ with finite capacity for the workload and server multiple vacations. He derives the joint stationary distribution of the workload, and the state of the server (available for service or on vacation). \cite{TakineHasegawa90} considered an $M/G/1$ queue with balking customers and deterministic deadline. They analyse this model when the deadline operates on the waiting time, and on sojourn time. The authors have obtained integral equations for the steady state probability distribution function of the waiting times and the sojourn times. They also expressed these equations in terms of steady state probability distribution function of the $M/G/1$ queue with vacation without deadline. More recently, \cite{Katayama2011} has investigated the $ M/G/1 $ queue with multiple and single vacation, sojourn time limits and balking behavior. Using the level crossing approach, explicit solutions for the stationary virtual waiting time distribution are derived under various assumptions on the service time distribution. 

Using the probability generating function approach,\cite{AltmanYechiali2006} and \cite{AltmanYechiali2008} have developed a comprehensive analysis, of queueing models such as $M/M/1+M$, $M/G/1+M$ $ G/M/1+M$, $M/M/c+M$ and $M/M/\infty + M$. 
They analyse these models when customers become impatient only when the server is on vacation. 
They obtain various performance measures, under both multiple and single vacation policies. 
\cite{Sudhesh2010} has considered the model $M/M/1$ with customer impatience and disastrous breakdown. The author obtains a close form expression for the time-dependent queue size distribution.
\cite{PerelYechiali2010} have investigated $M/M/c$ queues, for $c=1$, $1<c<\infty$ and $c=\infty$ where the impatience of customers id due to a slow service rate. The authors have derived the probability generating function of the queue-length and the mean queue size.
\cite{SelvarajuGoswami2013} considered an $M/M/1$ queueing system with impatience in which server serves at lower rate during a vacation period rather than completely stop serving. The authors derived some performance measures and the stochastic decomposition.
\cite{Sakuma2012} analysed the $M/M/c + D$ queue with multiple vacation exponentially distributed, where customers are impatient only when all servers are unavailable. He derives the stationary distribution of the system using the matrix-analytic method.
\cite{YueYueSaffer2014} considered a variant of the $M/M/1$ model introduce by \cite{AltmanYechiali2006}, they assumed that the server is allowed to take a maximum number $K$ of vacations if the system remains empty after the end of a vacation. 
They obtain closed form expressions for important performance measures.
\cite{AltmanYechiali2008} have investigate a $M/M/\infty$ queue system with single and multiple vacations in which customers are impatience only when servers are unavailable for service. 
The authors have derived the probability generating function of the number of customer in the system and some important performance measures.
\cite{Yechiali2007} has studied the $M/M/c$, for $1 \leq c \leq \infty$ queueing systems, where servers suffer disasters resulting in the loss all customers present in the system. Moreover, arriving customers during repair process are impatience. 
 \cite{EconomouKapodistria2010} extended the single server model by \cite{Yechiali2007} by assuming that customers perform synchronized abandonments.
Recently, \cite{YueYueZhao2015} extended the model $M/M/1+M$ under the single and multiple vacation policy studied by \cite{AltmanYechiali2006}. The authors assumed that customers are impatience regardless the state of the server. 
Furthermore, customers have a deadline depending on whether the server is on vacation or busy. 
Using the probability generating function approach, the authors derive some performance measures. 
In this paper, we analyse a Markovian queueing system where customer are also impatience regardless the state of the server, and the analysis is also based on the probability generating function.
Whereas we have completed this paper, soon after, \cite{YueYueZhao2015} had just been published. The main difference arises from the impatience assumption. Indeed, this paper we consider only one patience distribution, whereas Yue et al. consider two patience distributions depending on whether the server is 
on vacation or available for serving, resulting in different balance equations. 
Moreover, we consider also in this paper a multiserver queue with simple vacation.

The rest of paper is organized as follows. In Section~\ref{section1} and Section~\ref{section2} respectively, we analyse the model $M/M/1$ with exponentially distributed  patience time, under simple and multiple vacation policy. We derive the balance equations and we obtain and we solve the differential equations for PGFs of the steady-state probabilities. This enable us to calculate some performance measures such as the mean system sizes when the server is either on vacation or busy, the proportion of customers served, and the average sojourn time. In Section~\ref{section3} the model $M/M/c/V_{S} +M$.

\section{Analysis of \texorpdfstring{$M/M/1/V_S + M $}{Lg}}
\label{section1}

\subsection{Model description and notation}

We consider a queueing system in which customers arrive according to a Poisson process at rate $\lambda$. The service is provided by a single server of unlimited capacity. The successive service time of the server are assumed to be exponentially distributed with rate $\mu$. Whenever the system becomes empty, the server takes a vacation. If on return from vacation, the system is empty, the server waits for the next arrival, otherwise he begins service. The vacation times are exponentially distributed at rate $\gamma$. Assume furthermore that the customers are impatient. Whenever a customer arrives to the system, he becomes impatient regardless the states of the server. The impatient times are exponentially distributed at rate $\xi$. If a customer joins the queue, he stays until his service is completed.

The system is represented by a two-dimensional continuous Markov chain $\{ J(t), N(t) \}_{t \geq 0}$ with state space $S = \{ (j,n): j= 0,1, \; n=0,1,2,\ldots \}$, where $N(t)$ denotes the total number of customers present in the system and $J(t)$ indicates the state of the system at time $t$. If $J(t) = 0$, the server is on vacation, whereas if $J(t) = 1$, the server is serving customers or idle. The corresponding transition diagram is depicted in Figure~\ref{figureTransRate1}. Let us define the stationary probabilities for the Markov chain as

\begin{equation*}
p_{j,n} = \mathbb{P}(J(t) = j, N(t) = n), \quad j = 0,1, \quad n \geq 0.
\end{equation*}

\begin{figure}[h!]
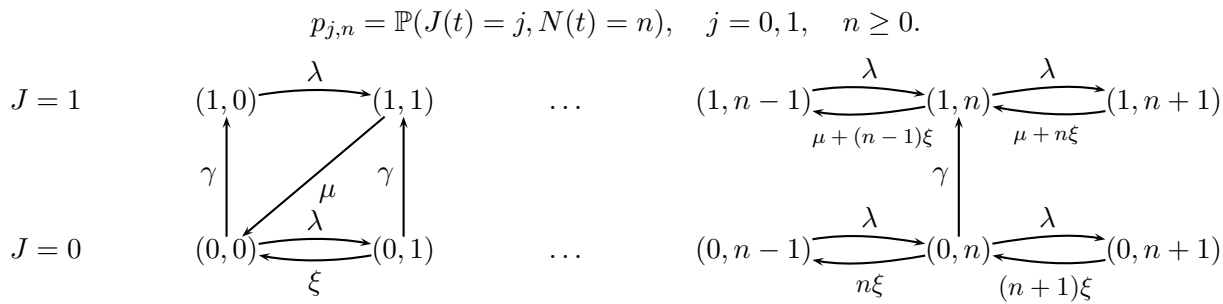

\begin{center}
\psmatrix[fillstyle=solid]
$J=1$ & $(1,0)$ & $(1,1)$ & $\ldots$ & $(1,n-1)$ & $(1,n)$ & $(1,n+1)$ \\
$J=0$ & $(0,0)$ & $(0,1)$ & $\ldots$ & $(0,n-1)$ & $(0,n)$ & $(0,n+1)$
\endpsmatrix
\psset{linecolor=black, arrows=->, labelsep=1mm,shortput=nab}

\ncarc[arcangle=10]{1,2}{1,3}^{$\lambda$}
\ncarc[arcangle=10]{1,5}{1,6}^{$\lambda$}
\ncarc[arcangle=10]{1,6}{1,7}^{$\lambda$}
\ncarc[arcangle=10]{1,6}{1,5}^{{\scriptsize $\mu + (n-1)\xi$}}
\ncarc[arcangle=10]{1,7}{1,6}^{{\scriptsize $\mu + n\xi$}}

\ncline{1,3}{2,2}^{$\mu$}
\ncarc[arcangle=10]{2,3}{2,2}^{$\xi$}
\ncarc[arcangle=10]{2,2}{2,3}^{$\lambda$}
\ncarc[arcangle=10]{2,5}{2,6}^{$\lambda$}
\ncarc[arcangle=10]{2,6}{2,7}^{$\lambda$}
\ncarc[arcangle=10]{2,6}{2,5}^{{\small $n\xi$}}
\ncarc[arcangle=10]{2,7}{2,6}^{{\small $(n+1)\xi$}}

\ncline{2,2}{1,2}^{$\gamma$}
\ncline{2,3}{1,3}^{$\gamma$}
\ncline{2,6}{1,6}^{$\gamma$}
\caption{Transition rate diagram of \texorpdfstring{$(J(t),N(t))$}{Lg} when the sever takes single vacation}
\label{figureTransRate1}
\end{center}
\end{figure}

\subsection{The equilibrium state distribution}

The set of balance equations is given as follows

\begin{equation}
\label{eq-balance00-VS}
 (\lambda + \gamma )p_{0,0} =  \xi p_{0,1} + \mu p_{1,1}.
\end{equation}

\begin{equation}
\label{eq-bamance0n-VS}
(\lambda + \gamma + n \xi)p_{0,n}= \lambda p_{0,n-1} + (n+1)\xi p_{0,n+1}.
\end{equation}

\begin{equation}
\label{eq-balance10-VS}
  \lambda p_{1,0} = \gamma p_{0,0}.
\end{equation}

\begin{equation}
\label{eq-balance1n-VS}
 (\lambda + \mu + (n-1)\xi)p_{1,n}= \lambda p_{1,n-1} + \gamma p_{0,n}+ (\mu + n\xi)p_{1,n+1}.
\end{equation}

It is convenient to introduce the probability generating function (PGFs) in order to solve Equations (\ref{eq-balance00-VS})-(\ref{eq-balance1n-VS}). Define the (partial) probability generating functions (PGFs) $P_{0}(z)$ and $P_{1}(z)$, for $0< z <1$ 

\begin{equation*}
P_{0}(z) = \sum_{n=0}^{\infty}z^n p_{0,n} \hspace{0.5cm} \mbox{and} \hspace{0.5cm} P_{1}(z)=\sum_{n=0}^{\infty}z^n p_{1,n},
\end{equation*}
with
\begin{equation}
\label{norming condition}
P_{0}(1) + P_{1}(1) = 1.
\end{equation}

Denote by $P'_{0}(z)$ and $P'_{1}(z)$ the respective derivatives of $P_{0}(z)$ and $P_{1}(z)$, for $0< z <1$

\begin{equation*}
P'_{0}(z) = \sum_{n=1}^{\infty}n z^{n-1} p_{0,n} \hspace{0.5cm} \mbox{and} \hspace{0.5cm} P'_{1}(z)= \sum_{n=1}^{\infty} n z^{n-1} p_{1,n}.
\end{equation*}

\n
Let us now discuss about several special cases:
\begin{enumerate}
\item $\xi = 0$. This case corresponds to a system where customers are not impatient. This model boils down to an $M/M/1/V_{s}$. This case was studied by \cite{ServiFinn2002}.

\item $\xi = 0$ only in Equation \eqref{eq-balance1n-VS}. This system corresponds to a system where customers are impatient only when the server is on vacation. This case was solved at first by \cite{AltmanYechiali2006} for $M/M/1$, $M/G/1$ and $M/M/c$ queues with multiple and single vacation.

\item $\gamma = 0$. This case represents a system without vacation. The state space is then one dimensional and the system boils down to an $M/M/1+M$ queue. A number of papers on queues with impatience phenomena have appeared. We may mention for example, for the $M/M/c$ queue \cite{Tijms1999} and the $M/M/1$ queue with priority \cite{BaraChung2001}.
\end{enumerate}

\n\\
Next, we will find the generating function $P_{0}$ and $P_{1}$. By multiplying \eqref{eq-balance00-VS} and \eqref{eq-bamance0n-VS} by $z^n$ and summing over $n$, we get, for $0 < z < 1$ 

\begin{equation}
\label{eq-diff-P0-VS}
\xi(1- z) P'_{0}(z) - \left[\lambda(1-z) + \gamma\right] P_{0}(z)= - \mu p_{1,1}.
 \end{equation}
 
\n
Equation \eqref{eq-diff-P0-VS} is similarly to Equation (5.5) in \cite{AltmanYechiali2006} and has been solved by the same author. We get in the same way a first order differential equation related to $P_{1}$, by using \eqref{eq-balance10-VS} and \eqref{eq-balance1n-VS}, we have

\begin{equation}
\label{eq-diff-P1-VS}
 (\lambda z -  \mu + \xi)(1-z) P_{1}(z) - \xi z(1-z)P'_{1}(z) = \gamma z P_{0}(z) + (\xi-\mu)(1-z) p_{1,0} - \mu z p_{1,1}.
\end{equation}

\n
We obtain a set a two first order linear differential equations involving the unknown probabilities $p_{1,0}$ and $p_{1,1}$. We shall see by solving (\ref{eq-diff-P0-VS}) that $p_{1,1}$ will be expressed in termes of $p_{0,0}$. Moreover, using Equation (\ref{eq-balance10-VS}), $p_{1,0}$ is expressed in termes of $p_{0,0}$. Thus, the PGFs $P_{0}$ and $P_{1}$ will be expressed in termes of $p_{0,0}$. Therefore, once $p_{0,0}$ will be calculated, $P_{0}$ and $P_{1}$ will completely determined.

\n\\
Dividing (\ref{eq-diff-P0-VS}) by $\xi(1-z)$ yields

\begin{equation}
\label{eq-diffbis-P0-VS}
P'_{0}(z) - \left[ \dfrac{\lambda}{\xi} + \dfrac{\gamma}{\xi(1-z)} \right] P_{0}(z) = -\dfrac{\mu}{\xi(1-z)}p_{1,1}.
\end{equation}

\n
To solve the above first order differential equation an integrating factor (see \cite{BoyceDiprima}) can be found as

\begin{equation*}
 e^{\int -\left(\frac{\lambda}{\xi} + \frac{\gamma}{\xi(1-z)}\right) dz} = e^{- \frac{\lambda}{\xi} z} (1-z)^{\frac{\gamma}{\xi}}.
\end{equation*}

\n
Multiplying both sides of (\ref{eq-diffbis-P0-VS}) by the integrating factor $e^{- \frac{\lambda}{\xi} z} (1-z)^{\frac{\gamma}{\xi}}$ we get

\begin{equation*}
\dfrac{d}{dz} \left[ e^{- \frac{\lambda}{\xi} z} (1-z)^{\frac{\gamma}{\xi}} P_{0}(z) \right] = - p_{1,1} \dfrac{\mu}{\xi} e^{- \frac{\lambda}{\xi}} (1-z)^{\frac{\gamma}{\xi}-1} .
\end{equation*}

\n
Integrating from $0$ to $z$ we have

\begin{equation*}
\label{solution equa diff0}
e^{ -\frac{\lambda}{\xi} z} (1-z)^{\frac{\gamma}{\xi}}P_{0}(z) - P_{0}(0) = - p_{1,1}\frac{\mu}{\xi} \int_{0}^{z}e^{ -\frac{\lambda}{\xi} s} (1-s)^{\frac{\gamma}{\xi} -1}ds,
\end{equation*}

\n
hence

\begin{equation}
\label{solution equa diff0-VS}
 P_{0}(z)  = e^{ \frac{\lambda}{\xi} z} (1-z)^{-\frac{\gamma}{\xi}}\left[ P_{0}(0) - p_{1,1}\frac{\mu}{\xi} \int_{0}^{z}e^{ -\frac{\lambda}{\xi} s} (1-s)^{\frac{\gamma}{\xi} -1}ds \right].
\end{equation}

\n
Since $P_{0}(1) = \sum_{n=0}^{\infty} p_{0,n} < \infty$ and $\lim_{z\rightarrow 1} (1-z)^{-\frac{\gamma}{\xi}} = \infty$, when taking the limite $z$ tending to $1$ in (\ref{solution equa diff0-VS}), and noting that $P_{0}(0) = p_{0,0}$ we expresse $p_{1,1}$ in termes of $p_{0,0}$

\begin{equation}
\label{p11 en fonction de p00}
p_{1,1} = \dfrac{\xi}{\mu A} p_{0,0}, 
\end{equation}
with

\begin{equation}
\label{A}
A(z) := \int_{0}^{z}e^{ -\frac{\lambda}{\xi} s} (1-s)^{\frac{\gamma}{\xi} -1}ds ,
\end{equation}
and
\begin{equation}
A := A(1) = \int_{0}^{1}e^{ -\frac{\lambda}{\xi} s} (1-s)^{\frac{\gamma}{\xi} -1}ds.
\end{equation}

\n
Note that $A(z)$ is well define if $\frac{\gamma}{\xi} - 1 >0$. Substituting (\ref{p11 en fonction de p00}) in (\ref{solution equa diff0-VS}) yields

\begin{equation}
\label{solution final equa diff0}
P_{0}(z) = p_{0,0} e^{ \frac{\lambda}{\xi} z} (1-z)^{-\frac{\gamma}{\xi}}\left[ 1 -  \dfrac{A(z)}{A} \right],
\end{equation}

\n
which is the as same Equation (2.12) in \cite{AltmanYechiali2006}, with $p_{1,1}$ given by (\ref{p11 en fonction de p00}). Equation (\ref{solution final equa diff0}) express $P_{0}(z)$ in terms of $p_{0,0}$. Thus, once $p_{0,0}$ is calculated $P_{0}(z)$ is completely determined. The next theorem focuses on the calculation of $p_{0,0}$. \vspace{0.3cm}

Next, we will find the generating function $P_{1}$. Dividing (\ref{eq-diff-P1-VS}) by $- \xi z(1-z)$ leads to

%

\begin{equation*}
P'_{1}(z) - \left[ \frac{\lambda}{\xi} - \left(\frac{\mu}{\xi} - 1\right)\frac{1}{z} \right]P_{1}(z) = -\dfrac{\gamma}{\xi(1-z)}P_{0}(z) -  \dfrac{1}{z}\left(1-\dfrac{\mu}{\xi}\right)p_{1,0} + \dfrac{\mu}{\xi(1-z)}p_{1,1}.
\end{equation*}

\n
Using Equation~(\ref{eq-balance10-VS}), as well as (\ref{p11 en fonction de p00}), the above equation is written as

\begin{equation}
\label{equa diff 1}
P'_{1}(z) - \left[ \frac{\lambda}{\xi} - \left(\frac{\mu}{\xi} - 1\right)\frac{1}{z} \right]P_{1}(z) = -\dfrac{\gamma}{\xi(1-z)}P_{0}(z) + \left[  \dfrac{(\mu - \xi)\gamma}{\xi \lambda z} + \dfrac{1}{(1-z)A}\right]p_{0,0}.
\end{equation}

\n
Multiplying both sides of (\ref{equa diff 1}) by $e^{-\frac{\lambda}{\xi}z} z^{\frac{\mu}{\xi} - 1}$ (with $\frac{\mu}{\xi} - 1 >0$) leads to

\begin{equation}
\frac{d}{dz} \left[ e^{-\frac{\lambda}{\xi}z} z^{\frac{\mu}{\xi} - 1}P_{1}(z) \right] = e^{-\frac{\lambda}{\xi}z} z^{\frac{\mu}{\xi} - 1}\left[ -\frac{\gamma}{\xi(1-z)}P_{0}(z) + p_{0,0}\left( \dfrac{(\mu - \xi)\gamma}{\xi \lambda z} + \frac{1}{A(1)(1-z)}  \right) \right],
\end{equation}

\n 
Integrating from $0$ to $z$, the solution to the differential equation is given by

\begin{equation*}
P_{1}(z) = e^{\frac{\lambda}{\xi}z} z^{-\left(\frac{\mu}{\xi} - 1\right)} \int_{0}^{z} e^{-\frac{\lambda}{\xi}z} z^{\frac{\mu}{\xi} - 1} \left[ -\frac{\gamma}{\xi(1-s)}P_{0}(s) + p_{0,0}\left\lbrace  \frac{(\mu-\xi )\gamma}{\lambda \xi s} + \frac{1}{A(1-s)} \right\rbrace \right]ds,
\end{equation*}
where $P_{0}(.)$ is given by (\ref{solution final equa diff0}). The above equation  is valid if and only if $\frac{\mu}{\xi} - 1$ which leads to $\mu > \xi$. Substituting (\ref{solution final equa diff0}) in the above equation, we have

\begin{equation}
\label{P1}
  P_{1}(z) = p_{0,0}e^{\frac{\lambda}{\xi}z} z^{-\left(\frac{\mu}{\xi} - 1\right)} \left[ -\frac{\gamma}{\xi} B(z) + \frac{(\mu - \xi)\gamma}{\xi \lambda} C(z) + \frac{ D(z)}{A}\right],
\end{equation}
where

\begin{equation}
\label{B}
B(z) = \int_{0}^{z} s^{\frac{\mu}{\xi} - 1 }(1-s)^{-\left(\frac{\gamma}{\xi}+1\right)} \left(1-\frac{A(s)}{A}\right) ds,
\end{equation}

\begin{equation}
\label{C}
C(z) = \int_{0}^{z} e^{-\frac{\lambda}{\xi}s} s^{\frac{\mu}{\xi} -2} ds,
\end{equation}

\n
and
\begin{equation}
\label{D}
D(z) = \int_{0}^{z} e^{-\frac{\lambda}{\xi}s} s^{\frac{\mu}{\xi} -1}(1-s)^{-1} ds.
\end{equation}

\n
All the previous is summarized in the following Proposition.

\begin{prop}
\label{prop1}
If $\rho := \frac{\lambda}{\mu} < 1$ and $ \frac{\xi}{\mu} < 1$ the partial generating function can be expressed in terms of $p_{0,0}$ as

\begin{equation}
\label{P0-VS}
P_{0}(z)  = p_{0,0} e^{\frac{\lambda}{\xi}z} (1-z)^{-\frac{\gamma}{\xi}} \left[ 1 - \dfrac{A(z)}{A} \right],
\end{equation}
and
\begin{equation}
\label{P1-VS}
  P_{1}(z) = p_{0,0}e^{\frac{\lambda}{\xi}z} z^{-(\frac{\mu}{\xi} - 1)} \left[ -\frac{\gamma}{\xi} B(z) + \frac{(\mu-\xi)\gamma}{\xi \lambda} C(z) + \frac{D(z)}{A} \right],
\end{equation}

\n
where, $A(.)$, $B(.)$, $C(.)$ and $D(.)$ are given respectively by (\ref{A}), (\ref{B}), (\ref{C}) and (\ref{D}).
\end{prop}

Equations (\ref{P0-VS}) and (\ref{P1-VS}) are expressed in terms of $p_{0,0}$, the probability that there is no customer in the system when the server is on vacation. In order to completely determined $P_{0}(.)$ and $P_{1}(.)$ we have to calculated $p_{0,0}$. The following Proposition gives this value.

\begin{prop}
\label{prop2}
If $\rho := \frac{\lambda}{\mu} < 1$ and $ \frac{\xi}{\mu} < 1$ the probability $p_{0,0}$ is given by

\begin{equation}
\label{p00}
p_{0,0} = \left[ \dfrac{\xi}{\gamma A} + e^{\frac{\lambda}{\xi}} \left\lbrace -\frac{\gamma}{\xi} B(1) + \frac{( \mu-\xi)\gamma}{\xi \lambda} C(1) + \frac{D(1)}{A} \right\rbrace \right]^{-1}.
\end{equation}
\end{prop}

\begin{proof}
The idea consists in expressing $P_{0}(1)$ and $P_{1}(1)$, which are respectively the probability that the server is on vacation and the probability that the server is serving customers or the server is ildle, in termes of $p_{0,0}$. In order to obtain $p_{0,0}$ we will use the norming condition (\ref{norming condition}). 

\n\\
From Equation (\ref{P0-VS}) by using the L'Hopital's rule  we get
 
\begin{equation*}
\label{P_0(1)_VS}
 P_{0}(1) = \lim_{\substack{z \rightarrow 1}}\dfrac{e^{\frac{\lambda}{\xi}} z p_{0,0} \left\lbrace \frac{\lambda}{\xi}\left(1-A^{-1}A(z) \right) - A^{-1}A'(z)\right\rbrace }{-\frac{\gamma}{\xi}(1-z)^{\frac{\gamma}{\xi}}-1},
\end{equation*} 
where $A'(z) = e^{-\frac{\lambda}{\xi}z} (1-z)^{\frac{\gamma}{\xi}-1}$ which yields 

\begin{equation}
\label{P_0(1)}
P_{0}(1) = p_{0,0} \frac{\xi}{\gamma A}.
\end{equation}

\n
Using (\ref{P1-VS}) for $z=1$ yields

\begin{equation*}
P_{1}(1) = 1-P_{0}(1) = p_{0,0}e^{\frac{\lambda}{\xi}} \left[  -\frac{\gamma}{\xi} B(1) - \frac{(\xi - \mu)\gamma}{\xi \lambda} C(1) + A^{-1}D(1) \right].
\end{equation*}

\n
Using the norming condition $P_{0}(1) + P_{1}(1) = 1$ we get


\begin{equation*}
p_{0,0} = \left[ \dfrac{\xi}{\gamma A} + e^{\frac{\lambda}{\xi}} \left\lbrace -\frac{\gamma}{\xi} B(1)- \frac{(\xi - \mu)\gamma}{\xi \lambda} C(1) + \frac{D(1)}{A}  \right\rbrace \right]^{-1}.
\end{equation*}

\end{proof}

\subsection{Performance measures}
One of the main performance measures is the sojourn times. The system performance measures of the model are given below.

\subsubsection{The probability that the server is on vacation.}

\begin{align*}
P_{vac} = \mathbb{P}(J=0) &= \sum_{n=0}^{\infty} \mathbb{P}(J=0,N=n)\\
                &= \sum_{n=0}^{\infty} p_{0,n} \\
                &= P_{0}(1) = p_{0,0} \frac{\xi}{\gamma A},
\end{align*}
with $p_{0,0}$ given by (\ref{p00}).

\subsubsection{The probability that the server is idle.}

\begin{equation*}
P_{idle} = \mathbb{P}(J=1,N=0) = p_{1,0} = \frac{\gamma}{\lambda}p_{0,0},
\end{equation*}
according (\ref{eq-balance10-VS}).

\subsubsection{The probability that the server is serving customers.}

\begin{align*}
P_{ser} &= \sum_{n=1}^{\infty} \mathbb{P}(J=1, N=n)\\
         &= \sum_{n=1}^{\infty} p_{1,n}\\
         &= 1 - P_{idle} - P_{vac}\\
         &= 1 - p_{0,0}\left[ \frac{\gamma}{\lambda} - \frac{\xi}{\gamma A(1)} \right].
\end{align*}

For $j=0,1$, let $N_{j}$ be the system size when the server is in state $j$. Then $\mathbb{E}(N_{j})$ is the mean system size when the server is in state $j$.

\subsubsection{The mean number of customers when the system is on vacation}

\begin{align*}
\mathbb{E}(N_{0}) = \sum_{n=0}^{\infty} n \mathbb{P}(J=0,N=n) = \sum_{n=1}^{\infty} n p_{0,n} = P^{'}_{0}(1)
\end{align*}

\n
From (\ref{eq-diffbis-P0-VS}) and by using the Hospital's rule we get

\begin{align*}
P'_{0}(1) &= \lim_{z \rightarrow 1} \dfrac{\left[ \lambda(1-z) +\gamma \right]P_{0(z)} - \mu p_{11}}{\xi (1-z) }\\
             &= \lim_{z \rightarrow 1} \dfrac{-\lambda P_{0}(z)+ \left[ \lambda(1-z) +\gamma \right]P^{'}_{0}(z)}{-\xi}\\
             &= \dfrac{-\lambda P_{0}(1)+\gamma P^{'}_{0}(1)}{-\xi}.
             \end{align*}

\n
Thus we get,

\begin{equation}
P'_{0}(1) = \frac{\lambda}{\gamma + \xi}P_{0}(1) = \frac{\lambda}{\gamma + \xi} \frac{\xi}{\gamma A}p_{0,0} 
\end{equation}

We now derive the mean number of customers during non-vacation period. From (\ref{eq-diff-P1-VS}) we have

\begin{equation*}
\mathbb{E}(N_1)= P'_{1}(z) = \dfrac{(\lambda z - \mu + \xi)(1-z) - \gamma z P_{0}(z) - (\xi - \mu)(1-z)p_{1,0} + \mu z p_{1,1}}{\xi z (1-z)}.
\end{equation*}

\n Using the Hopital's rule we get

\begin{align*}
P'_{1}(1) &= \lim_{z\rightarrow 1} \dfrac{ \lambda(1-z) - (\lambda z - \mu + \xi) - \gamma P_{0}(z) - \gamma z P'_{0}(z)+(\xi - \mu) p_{1,0} + \mu p_{1,1}}{\xi(1-2z)}\\
             &= \dfrac{ - (\lambda  - \mu + \xi) - \gamma P_{0}(1) - \gamma P'_{0}(1)+(\xi - \mu) p_{1,0} + \mu p_{1,1}}{-\xi}\\
             &= \left[ -(\lambda - \mu + \xi) - \frac{\lambda}{\lambda + \xi}\frac{\xi}{A}p_{0,0} + \frac{\xi \gamma}{\lambda}p_{0,0} - \frac{\mu \gamma}{\lambda}p_{0,0} \right]\times\frac{1}{- \xi}.
\end{align*}

\n
Thus we have

\begin{equation}
\mathbb{E}(N_1) = \frac{\lambda - \mu + \xi}{\xi} + \left[ \frac{\lambda}{(\lambda + \xi)A}-\frac{\gamma}{\lambda} + \frac{\mu \gamma}{\lambda \xi} \right]p_{0,0},
\end{equation}
where $p_{0,0}$ is given by (\ref{p00}).

\subsubsection{Sojourn times}
\label{section-perf-sojour time}

Let $S$ be the unconditional total sojourn time of an arbitrary customer in the system, regardless of whether he completes service or not. Denote $S_{j,n} := \mathbb{E}(S\mid (X_{0}=(j,n+1)))$ the conditional sojourn time of a tagged customer in the system who does not abandon the system, given that the state upon his arrival is $(j,n)$, because the tagged customer is included in the system. We use first-step analysis method which consists by considering what the Markov chain dos at time $1$, i.e. after it takes on step from its current position. The total sojourn time when the tagged customer upon his arrival is $(1,0)$, i.e. the server is idle, is given by

\begin{equation}
\mathbb{E}(S_{1,0}) = \frac{1}{\mu}.
\end{equation}

\n
Let us calculate $\mathbb{E}(S_{1,n})$.
By conditioning on whether the next transition is a departure (either because of completion service or an impatient customer) or an arrival, we obtain

\begin{align*}
\begin{split}
\mathbb{E}(S_{1,n}) &= \mathbb{E}\left(S\mid X_{0}=(1,n+1)\right) \\
    &= \frac{\mu + n\xi}{\beta_n}\mathbb{E}\left( S \mid X_{0}=(1,n+1), X_{1}=(1,n)\right) + \frac{\lambda}{\beta_{n}}\mathbb{E}\left( S \mid X_{0}=(1,n+1), X_{1}=(1,n+2)\right)  \\
    &=  \frac{\mu + n\xi}{\beta_n} \left( \frac{1}{\beta_{n}} + \mathbb{E}(S \mid X_{0}=(1,n)) \right) + \frac{\lambda}{\beta_{n}}\left( \frac{1}{\beta_{n}} + \mathbb{E}(S \mid X_{0}=(1,n+1)) \right)\\
    &= \frac{1}{\beta_n} + \frac{\mu + n\xi}{\beta_n} \mathbb{E}(S \mid X_{0}=(1,n)) + \frac{\lambda}{\beta_{n}}\mathbb{E}(S_{1,n}).
\end{split}
\end{align*}
where $\beta_{n} = \lambda + \mu + n\xi$. Since 

\begin{align*}
\frac{\mu + n\xi}{\beta_n} \mathbb{E}(S \mid X_{0}=(1,n)) &= \frac{\mu}{\beta_{n}}\mathbb{E}(S_{1,n-1}) + \frac{n\xi}{\beta_{n}}\mathbb{E}(S_{1,n-1})\\
         &= \frac{\mu}{\beta_{n}}\mathbb{E}(S_{1,n-1}) + \frac{n\xi}{\beta_{n}} \left( \frac{1}{n}\times 0 + \frac{n-1}{n}\mathbb{E}(S_{1,n-1})  \right),
\end{align*}
we get

\begin{equation*}
\mathbb{E}(S_{1,n}) = \frac{1}{\beta_n}+ \frac{\mu + (n-1)\xi}{\beta_n} \mathbb{E}(S_{1,n-1}) + \frac{\lambda}{\beta_{n}}\mathbb{E}(S_{1,n}),
\end{equation*}
thus

\begin{equation}
\label{eq rec E(S_{1,n})}
\mathbb{E}(S_{1,n}) = \frac{1}{\mu + n\xi} + \frac{\mu + (n-1)\xi}{\beta_n}\mathbb{E}(S_{1,n}).
\end{equation}

\n
Iterating Equation~\eqref{eq rec E(S_{1,n})} we get for $n \geq 0$

\begin{equation}
\label{E(S_{1,n})}
\mathbb{E}(S_{1,n})= \frac{n+1}{\mu + n\xi}.
\end{equation}

\begin{rem}
If $\xi=0$, Equation~\eqref{E(S_{1,n})} becomes $\mathbb{E}(S_{1,n})= \frac{n+1}{\mu}$, which is the total sojourn time of the tagged customer when customers are not subject to impatience.
\end{rem}

The calculation of $\mathbb{E}(S_{0,0})$ and $\mathbb{E}(S_{0,n})$ is also based on the first step analysis. When the tagged customer finds the system upon arrival in the state $(0,0)$, the next transition is an arrival, a vacation completion or a the departure of the tagged customer, thus by conditioning on these events we obtain

\begin{align*}
\begin{split}
   \mathbb{E}(S_{0,0}) &= \frac{\gamma}{\lambda + \mu + \xi}\mathbb{E}(S\mid X_{0}=(0,1), X_{1}=(1,1)) + \frac{\lambda}{\lambda + \mu + \xi}\mathbb{E}(S\mid X_{0}=(0,1), X_{1}=(0,2))  \\
   & \qquad+ \frac{\xi}{\lambda + \mu + \xi}\mathbb{E}(S\mid X_{0}=(0,1), X_{1}=(0,0)),
\end{split}
\end{align*}
thus,

\begin{equation}
\label{E(S_{0,0})}
\mathbb{E}(S_{0,0}) = \frac{\gamma}{\gamma + \xi}\left( 1 + \frac{\gamma}{\mu} \right).
\end{equation}

\n
and for $n \geq 1$

\begin{equation*}
\mathbb{E}(S_{0,n}) = \frac{1}{\gamma + (n+1)\xi} + \frac{\gamma}{\gamma + (n+1)\xi}\mathbb{E}(S_{1,n}) + \frac{n\xi}{\gamma + (n+1)\xi} \mathbb{E}(S_{0,n-1}).
\end{equation*}

\n
Subsituting \eqref{E(S_{1,n})} in the above Equation yields

\begin{align}
\label{E(S_{0,n})}
\sum_{k=0}^{n} \frac{\xi^{n-k}}{\prod_{j=k+1}^{n+1}(\gamma + j\xi)} \frac{n!}{k!}\left( \frac{(k+1)\gamma}{\mu + n\xi} +1\right).
\end{align}

\section{Analysis of \texorpdfstring{$M/M/1/V_M + M $}{Lg}}
\label{section2}

\subsection{Balance equations}

We consider now the $M/M/1$ queue where the server takes multiple vacations at the end of a busy period. If the server returns from a vacation to an empty system he takes another vacation and so on. Otherwise, if he begins service until there is no customers. Customer as before are impatient. The system state, as in Section~\ref{section1} can be represented by a continuous Markov chain $\{ J(t),N(t) \}_{t \geq 0}$. The corresponding transition diagram is the same as Figure\ref{figureTransRate1} but without the state $(1,0)$ since the server is never idle (see Figure~\ref{figureTransRate2}). 

\begin{figure}[h!]
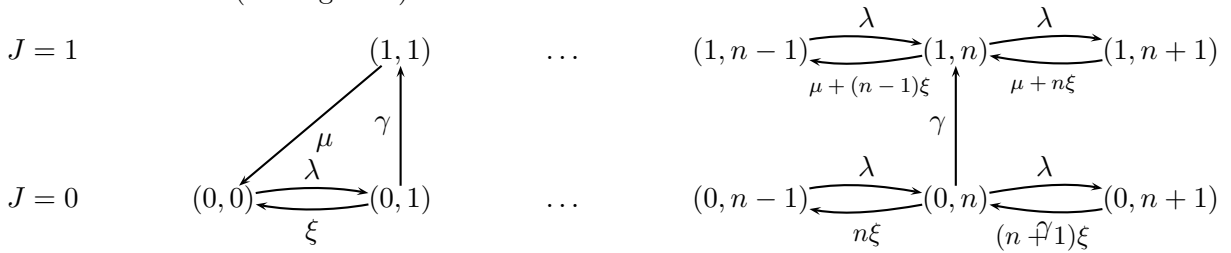

\begin{center}
\psmatrix[fillstyle=solid]
$J=1$ &         & $(1,1)$ & $\ldots$ & $(1,n-1)$ & $(1,n)$ & $(1,n+1)$ \\
$J=0$ & $(0,0)$ & $(0,1)$ & $\ldots$ & $(0,n-1)$ & $(0,n)$ & $(0,n+1)$
\endpsmatrix
\psset{linecolor=black, arrows=->, labelsep=1mm,shortput=nab}

\ncarc[arcangle=10]{1,2}{1,3}^{$\lambda$}
\ncarc[arcangle=10]{1,5}{1,6}^{$\lambda$}
\ncarc[arcangle=10]{1,6}{1,7}^{$\lambda$}
\ncarc[arcangle=10]{1,6}{1,5}^{{\scriptsize $\mu + (n-1)\xi$}}
\ncarc[arcangle=10]{1,7}{1,6}^{{\scriptsize $\mu + n\xi$}}

\ncline{1,3}{2,2}^{$\mu$}
\ncarc[arcangle=10]{2,3}{2,2}^{$\xi$}
\ncarc[arcangle=10]{2,2}{2,3}^{$\lambda$}
\ncarc[arcangle=10]{2,5}{2,6}^{$\lambda$}
\ncarc[arcangle=10]{2,6}{2,7}^{$\lambda$}
\ncarc[arcangle=10]{2,6}{2,5}^{{\small $n\xi$}}
\ncarc[arcangle=10]{2,7}{2,6}^{{\small $(n+1)\xi$}}

\ncline{2,2}{1,2}^{$\gamma$}
\ncline{2,3}{1,3}^{$\gamma$}
\ncline{2,6}{1,6}^{$\gamma$}
\caption{Transition rate diagram of \texorpdfstring{$(J(t),N(t))$}{Lg} when the server takes multiple vacations}
\label{figureTransRate2}
\end{center}
\end{figure}

Let us define the stationary probabilities $\{ p_{j,n}, j=0,1 \; n=0,1,\ldots \}$ and the partial generating functions, $P_{0}(z)$ and $P_{1}(z)$, for $0<z<1$

\begin{equation*}
P_{0}(z) = \sum_{n=0}^{\infty}z^n p_{0,n} \hspace{1cm} \mbox{and} \;\; P_{1}(z)=\sum_{n=1}^{\infty}z^n p_{1,n},
\end{equation*}
with
\begin{equation*}
P_{0}(1) + P_{1}(1) = 1.
\end{equation*}

The balance equations of the model are given as follows

\begin{equation}
\label{eq-balance00-VM}
 \lambda p_{0,0} =  \xi p_{0,1} + \mu p_{1,1}.
\end{equation}

\begin{equation}
\label{eq-balance0n-VM}
 (\lambda + \gamma + n \xi)p_{0,n}= \lambda p_{0,n-1} + (n+1)\xi p_{0,n+1},\qquad  n\geq 1.
\end{equation}

\begin{equation}
\label{eq-balance11-VM}
   (\lambda + \mu) p_{1,1} = (\mu + \xi) p_{1,2} + \gamma p_{0,1}.
\end{equation}

\begin{equation}
\label{eq-balance1n-VM}
 (\lambda + \mu + (n-1)\xi)p_{1,n}= \lambda p_{1,n-1} + \gamma p_{0,n}+ (\mu + n\xi)p_{1,n+1},\qquad  n\geq 2 .
\end{equation}

\begin{rem}
Equations~\eqref{eq-balance00-VM} and \eqref{eq-balance0n-VM} are the same as Equation (2.1) in \cite{AltmanYechiali2006}.
\end{rem}

Multiplying each Equations \eqref{eq-balance00-VM}-\eqref{eq-balance1n-VM} by $z^{n}$ and adding over $n$ yields the following linear first order differential equations

\begin{equation}
\label{equa-diff0-VM}
 \xi (1-z)P'_{0}(z) - \left[ \lambda(1-z) + \gamma \right]P_{0}(z) =  -(\gamma p_{0,0}+\mu p_{1,1}).
\end{equation}

\begin{equation}
\label{equa-diff1-VM}
\left[ (\lambda z - \mu)(1-z) + \xi(1-z) \right]P_{1}(z) - \xi z(1-z)P'_{1}(z) = \gamma zP_{0}(z) - (\gamma p_{0,0} + \mu p_{1,1} ) + \xi p_{1,1}.
\end{equation}

\begin{rem}
When Comparing Equation~(\ref{equa-diff0-VM}) with (\ref{eq-diff-P0-VS}) for the single vacation case, we have add the constant $\gamma p_{0,0}$. Note that this equation is the same as Equation (2.3) \cite{AltmanYechiali2006}.
\end{rem}

\subsection{Solution of differential equations}

The differential equation has been solved in \cite{AltmanYechiali2006}. The solution is obtained using the same method as in previous section. We multiply both sides of Equation~\eqref{equa-diff0-VM} by the integrating factor $e^{-\lambda/\xi z}(1-z)^{\gamma/\xi}$, which yields

\begin{equation}
\label{solution P0-VM}
P_{0}(z) =e^{\frac{\lambda}{\xi} z}(1-z)^{-\frac{\gamma}{\xi}}\left[P_{0}(0) -\frac{C_{\lambda, \mu}}{\xi}\int_{s=0}^{z}e^{- \frac{\lambda}{\xi} z}(1-s)^{\frac{\gamma}{\xi} -1}ds \right],
\end{equation}
where $C_{\lambda, \mu} = \mu p_{1,1} + \gamma p_{0,0}$. Taking the limit as $z \rightarrow 1$ in the above equation and using the fact that $\lim_{z \rightarrow 1}(1-z)^{-\frac{\gamma}{\xi}} = +\infty$ yields

\begin{equation}
\label{p00-en-fct-de-p11}
P_{0}(0) = P_{0,0} = \frac{C_{\lambda, \mu}}{\xi}A,
\end{equation}
where $A(.)$ is given by \eqref{A}. Substituting $C_{\lambda, \mu} = \frac{\xi}{A(1)} p_{0,0}$ in Equation~\eqref{solution P0-VM} gives

\begin{equation}
\label{P0-VM}
P_{0}(z) =e^{\lambda/\xi z}(1-z)^{-\gamma/\xi}p_{0,0}\left[1 -\frac{A(z)}{A(1)} \right].
\end{equation}

Next we will find the generating function $P_{1}(.)$ by solving Equation~\eqref{equa-diff1-VM} which can be written as

\begin{equation*}
P_{1}'(z) - \left[ \frac{\lambda}{\xi} - \frac{1}{z}\left( \frac{\mu}{\xi} - 1\right) \right]P_{1}(z) = -\frac{\gamma}{\xi(1-z)}P_{0}(z) + \frac{\gamma}{\xi z (1-z)}p_{0,0} + \frac{\mu - \xi}{\xi z(1-z)}p_{1,1}.
\end{equation*}

\n
Multiplying both sides of the above equation by $e^{- \frac{\lambda}{\xi} z} z^{\frac{\mu}{\xi}-1}$ and integrating from $0$ to $z$

\begin{equation}
\label{P1-VM}
e^{-\frac{\lambda}{\xi} z} z^{\frac{\mu}{\xi}-1} P_{1}(z) = \int_{0}^{z} e^{-\frac{\lambda}{\xi} s} s^{\frac{\mu}{\xi}-1}\left[ -\frac{\gamma}{\xi(1-s)}P_{0}(s) + \frac{\gamma}{\xi s (1-s)}p_{0,0} + \frac{\mu - \xi}{\xi s(1-s)}p_{1,1} \right]ds.
\end{equation}

\n
Using the definition of $C_{\lambda,\mu}$ and Equation~\eqref{p00-en-fct-de-p11} we have 

\begin{equation}
\label{p11 en fonction de p00-VM}
 p_{1,1} = \frac{\xi - \gamma A}{\mu A} p_{0,0}
\end{equation}
Equations~\eqref{P0-VM} and \eqref{P1-VM} are expressed in termes of $P_{0,0}$. Therefore, once $p_{0,0}$ are calculated, $P_{0}(.)$ and $P_{1}(.)$ are completely determined.


\subsection{Calculation of \texorpdfstring{$p_{0,0}$}{Lg}}

In order to calculate the value of $p_{0,0}$ we have to evaluate the quantities $P_{0}(1)$ and $P_{1}(1)$ in termes of $P_{0,0}$ and using the norming condition $P_{0}(1) + P_{1}(1) = 1$. Using~(\ref{equa-diff0-VM}) with $z=1$ and Equation~\eqref{p11 en fonction de p00-VM} we get

\begin{equation}
P_{0}(1) = \frac{C_{\lambda, \mu}}{\gamma}.
\end{equation}
We have $C_{\lambda, \mu} = \frac{p_{0,0}\xi}{A}$, thus 
\begin{equation}
\label{P_0(1) en fonction  de p00-VM}
P_{0}(1) = \frac{\xi}{A\gamma}p_{0,0}.
\end{equation}

\n
Substituting Equation~\eqref{P0-VM} in \eqref{P1-VM} we obtain

\begin{equation*}
P_{1}(z) = e^{\frac{\lambda}{\xi}z}z^{1-\frac{\mu}{\xi}}\left[ -p_{0,0}\frac{\gamma}{\xi}B(z) + \frac{\gamma}{\xi}p_{0,0}E(z) + p_{1,1}\left( \frac{\mu}{\xi} -1\right) E(z) \right],
\end{equation*}
where $B(.)$ is given by \eqref{B} and $E(z) = \int_{0}^{z} e^{-\frac{\lambda}{\xi}s}s^{\frac{\mu}{\xi} -2}(1-s)^{-1} ds$. Using the above equation for $z=1$ with \eqref{p11 en fonction de p00-VM} we get

\begin{equation}
\label{P_1(1) en fonction de p00-VM}
P_{1}(1) = p_{0,0} e^{\frac{\lambda}{\xi}}\left[ -\frac{\gamma}{\xi}B(1) + \left( \frac{1}{A} - \frac{\xi}{A\mu} + \frac{\gamma}{\mu}\right) E(1) \right].
\end{equation}

\n
Substituting Equations~\eqref{P_0(1) en fonction  de p00-VM} and \eqref{P_1(1) en fonction de p00-VM} in the norming condition yields

\begin{equation}
\label{p00-VM} 
p_{0,0} = \left[ \frac{\xi}{A\gamma} + e^{\frac{\lambda}{\xi}}\left(-\frac{\gamma}{\xi}B + \left( \frac{1}{A} - \frac{\xi}{A\mu} + \frac{\gamma}{\mu}\right)E \right) \right]^{-1}
\end{equation}


\subsection{Performance measures}

\subsubsection{Probability that the server is on vacation and probability that the server is busy}

The probability that the server is on vacation is
\begin{equation*}
P_{vac} = \mathbb{P}(J=0) = P_{0}(1) = p_{0,0}\frac{\xi}{\gamma A}.
\end{equation*}

The probability that the server is serving customers is
\begin{equation}
P_{ser} = \mathbb{P}(J=1) = 1 - P_{vac} = 1 - p_{0,0}\frac{\xi}{\gamma A},
\end{equation}
where $p_{0,0}$ is given by \eqref{p00-VM}.

\subsubsection{Mean number of customers when the server in on vacation and busy}

As given by \cite{AltmanYechiali2006}( Equation 2.16), the mean number od customers whenthe serveur is busy is

\begin{equation}
 \mathbb{E}(N_{0}) = P'_{0}(1) = \frac{\lambda}{\gamma + \xi} P_{0}(1).
\end{equation}

Next we derive the mean number of customers in the system when the server is busy. From Equation~\eqref{equa-diff1-VM} we have using the l'Hopital's rule

\begin{align*}
\begin{split}
 P'_{1}(1) &= \lim_{z \rightarrow 1} \dfrac{\left[\lambda(1-2z)+\mu  -\xi\right]P_{1}(z) + \left[(\lambda z - \mu)(1-z) + \xi(1-z) \right]P'_{1}(z)- \gamma P_{0}(z) - \gamma z P'_{0}(z) }{\xi(1-2z)}\\
        &= \dfrac{\left[\lambda(1-2)+\mu  -\xi\right]P_{1}(1) - \gamma P_{0}(1) - \gamma P'_{0}(1) }{\xi(1-2z)}
\end{split}\\
        &= \dfrac{(-\lambda + \mu - \xi)P_{1}(1)- \gamma P_{0}(1)} {\xi - \gamma},
\end{align*}
where $P_{0}(1)$ and $P_{1}(1)$ are given respectively by \eqref{P_0(1) en fonction  de p00-VM} and \eqref{P_1(1) en fonction de p00-VM}.

\subsubsection{Sojourn time}
Let $S$ and $S_{j,n}$ be defined as in section~\ref{section-perf-sojour time}. The expression $\mathbb{E}(S_{1,n})$ is the same as Equation~\eqref{E(S_{1,n})} but for $n \geq 1$ and $\mathbb{E}(S_{0,n})$ is given by Equation~\eqref{E(S_{0,n})}.

\section{Analysis of \texorpdfstring{$M/M/c/V_S + M $, $1 \leq c < \infty$}{Lg}}
\label{section3}

\subsection{Model description and notation}

We consider a multi-server queueing system $M/M/c$ with $c$ identical servers and an unlimited single waiting line. Customers arrive at the system according to a Poisson process with rate $\lambda$ and are served with an FCFS (first come, first served) discipline. The service times of the customers are independent exponentially distributed random variables with common mean $\mu^{-1}$. If a server find an empty waiting line at service completion instant, it takes single vacation period. The vacation durations of servers are exponentially distributed random variables with mean $\gamma^{-1}$. Assume furthermore the customers are impatient regardless the state of the server. If the service of a customer is not started before its patience runs out, then he leaves the system and never returns. The patience times of the customers are i.i.d. random variables exponentially distributed with mean $\xi^{-1}$.

The system is again represented by a two-dimensional continuous Markov process $\{ J(t),N(t) \}_{t \geq 0}$ with state space $\mathcal{S} = \{(j,n): j=0,1,\ldots , c, \, n \geq 0 \}$. As the previous study, $N(t)$ represents the number of customers present in the system at time $t$, but here $J(t)$ is the number of operating servers (busy or Idle). The stationary probabilities of the Markov process are define by

\begin{equation*}
  p_{j,n} = \mathbb{P}(J(t) = j, N(t) = n), \; j=0, \ldots , c , \;\; n \geq 0.
\end{equation*}

\subsection{Equilibrium state distribution}

For $J=0$, i.e. all the servers are on vacation.  The (partial) transition-rate diagram is depicted in Figure~\ref{figureTransRate3}, and the balance equations are given as follow

\begin{equation}
\label{eq-balance0c}
 \begin{array}{lll}
              n = 0 \quad (\lambda + c\gamma )p_{0,0} =  \xi p_{0,1} + \mu p_{1,1} \\ \\
              n\geq 1 \quad (\lambda + c\gamma + n \xi)p_{0,n}= \lambda p_{0,n-1} + (n+1)\xi p_{0,n+1}.
                 \end{array}          
\end{equation}

\begin{figure}[h!]
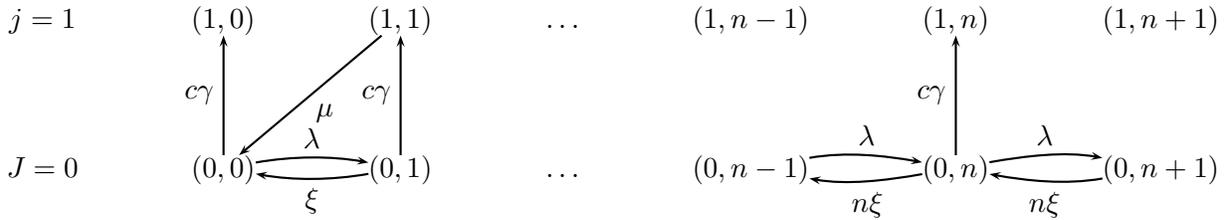

\begin{center}
\psmatrix[fillstyle=solid]
$j=1$ & $(1,0)$   & $(1,1)$   & $\ldots$ & $(1,n-1)$ & $(1,n)$ & $(1,n+1)$ \\
$J=0$ & $(0,0)$   & $(0,1)$   & $\ldots$ & $(0,n-1)$ & $(0,n)$ & $(0,n+1)$
\endpsmatrix
\psset{linecolor=black, arrows=->, labelsep=1mm,shortput=nab}

\ncarc[arcangle=10]{2,2}{2,3}^{$\lambda$}
\ncline{2,2}{1,2}^{$c\gamma$}
\ncarc[arcangle=10]{2,3}{2,2}^{$\xi$}
\ncline{2,3}{1,3}^{$c\gamma$}
\ncline{1,3}{2,2}^{$\mu$}

\ncarc[arcangle=10]{2,6}{2,7}^{$\lambda$}
\ncline{2,6}{1,6}^{$c\gamma$}
\ncarc[arcangle=10]{2,6}{2,5}^{$n\xi$}
\ncarc[arcangle=10]{2,5}{2,6}^{$\lambda$}
\ncarc[arcangle=10]{2,7}{2,6}^{$n\xi$}

\caption{Transition rate diagram of \texorpdfstring{$(J(t),N(t))$}{Lg} when the servers take single vacation}
\label{figureTransRate3}
\end{center}
\end{figure}

\n\\
For $J=j$ with $1 \leq j \leq c-1$, i.e. $j$ servers are in operating state (serving cutomers or idle waiting for service). The transition rate-diagram is depicted in Figure~\ref{figureTransRate4} for $0 \leq j \leq 3$ and the balance equations are given by

\begin{figure}[h!]
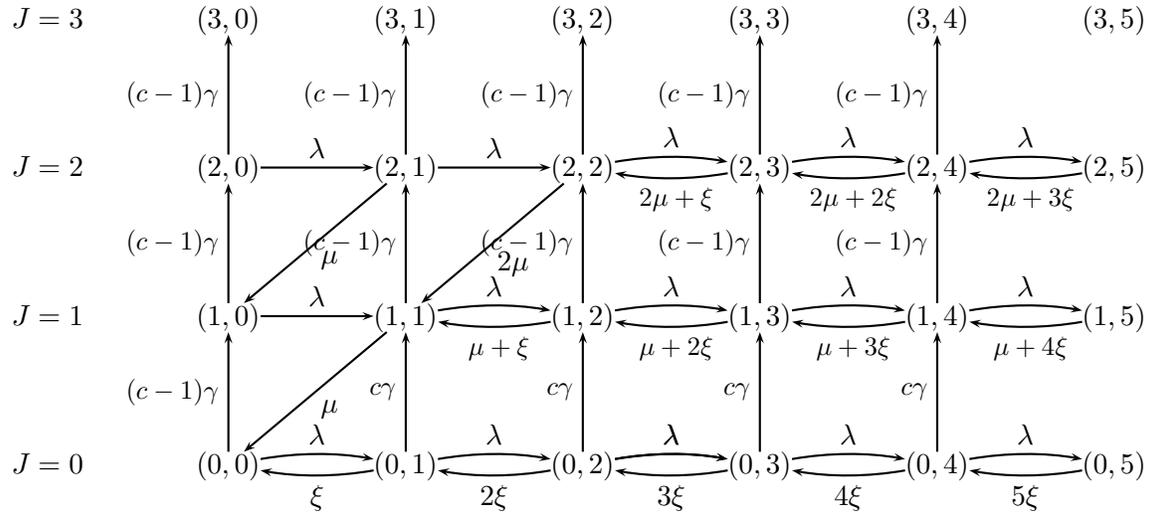

\begin{center}
\psmatrix[fillstyle=solid]
$J=3$ &  $(3,0)$   & $(3,1)$   & $(3,2)$ & $(3,3)$ & $(3,4)$ & $(3,5)$ \\
$J=2$ &  $(2,0)$   & $(2,1)$   & $(2,2)$ & $(2,3)$ & $(2,4)$ & $(2,5)$ \\
$J=1$ & $(1,0)$   & $(1,1)$   & $(1,2)$ & $(1,3)$ & $(1,4)$ & $(1,5)$ \\
$J=0$ & $(0,0)$   & $(0,1)$   & $(0,2)$ & $(0,3)$ & $(0,4)$ & $(0,5)$ 
\endpsmatrix
\psset{linecolor=black, arrows=->, labelsep=1mm,shortput=nab}

\ncarc[arcangle=10]{4,2}{4,3}^{$\lambda$}
\ncline{4,2}{3,2}^{{\small $(c-1)\gamma$}}

\ncarc[arcangle=10]{4,3}{4,2}^{$\xi$}
\ncline{4,3}{3,3}^{{\small $c\gamma$}}
\ncarc[arcangle=10]{4,4}{4,5}^{$\lambda$}
\ncarc[arcangle=10]{4,3}{4,4}^{$\lambda$}

\ncarc[arcangle=10]{4,4}{4,5}^{$\lambda$}
\ncarc[arcangle=10]{4,4}{4,3}^{$2\xi$}
\ncline{4,4}{3,4}^{{\small $c\gamma$}}
\ncarc[arcangle=10]{4,5}{4,4}^{$3\xi$}

\ncline{4,5}{3,5}^{{\small $c\gamma$}}
\ncarc[arcangle=10]{4,5}{4,6}^{$\lambda$}
\ncarc[arcangle=10]{4,6}{4,5}^{$4\xi$}

\ncline{4,6}{3,6}^{{\small $c\gamma$}}
\ncarc[arcangle=10]{4,6}{4,7}^{$\lambda$}
\ncarc[arcangle=10]{4,7}{4,6}^{$5\xi$}

\ncline{3,2}{3,3}^{$\lambda$}

\ncarc[arcangle=10]{3,3}{3,4}^{$\lambda$}
\ncarc[arcangle=10]{3,4}{3,5}^{$\lambda$}
\ncarc[arcangle=10]{3,5}{3,6}^{$\lambda$}
\ncarc[arcangle=10]{3,6}{3,7}^{$\lambda$}
\ncline{3,3}{4,2}^{$\mu$}
\ncarc[arcangle=10]{3,4}{3,3}^{{ \small $\mu +\xi$}}
\ncarc[arcangle=10]{3,5}{3,4}^{{ \small $\mu +2\xi$}}
\ncarc[arcangle=10]{3,6}{3,5}^{{ \small $\mu +3\xi$}}
\ncarc[arcangle=10]{3,7}{3,6}^{{ \small $\mu +4\xi$}}
\ncline{3,2}{2,2}^{{\small $(c-1)\gamma$}}
\ncline{3,3}{2,3}^{{\small $(c-1)\gamma$}}
\ncline{3,4}{2,4}^{{\small $(c-1)\gamma$}}
\ncline{3,5}{2,5}^{{\small $(c-1)\gamma$}}
\ncline{3,6}{2,6}^{{\small $(c-1)\gamma$}}

\ncline{2,2}{2,3}^{$\lambda$}
\ncline{2,3}{2,4}^{$\lambda$}
\ncarc[arcangle=10]{2,4}{2,5}^{$\lambda$}
\ncarc[arcangle=10]{2,5}{2,6}^{$\lambda$}
\ncarc[arcangle=10]{2,6}{2,7}^{$\lambda$}
\ncline{2,3}{3,2}^{$\mu$}
\ncline{2,4}{3,3}^{$2\mu$}
\ncarc[arcangle=10]{2,5}{2,4}^{{ \small $2\mu +\xi$}}
\ncarc[arcangle=10]{2,6}{2,5}^{{ \small $2\mu +2\xi$}}
\ncarc[arcangle=10]{2,7}{2,6}^{{ \small $2\mu +3\xi$}}
\ncline{2,2}{1,2}^{{\small $(c-1)\gamma$}}
\ncline{2,3}{1,3}^{{\small $(c-1)\gamma$}}
\ncline{2,4}{1,4}^{{\small $(c-1)\gamma$}}
\ncline{2,5}{1,5}^{{\small $(c-1)\gamma$}}
\ncline{2,6}{1,6}^{{\small $(c-1)\gamma$}}
\caption{Transition rate diagram of \texorpdfstring{$(J(t),N(t))$}{Lg} when the servers take single vacation}
\label{figureTransRate4}
\end{center}
\end{figure}

\begin{equation}
\label{eq-balancejc}
\begin{array}{lllll}
                    n = 0 \quad (\lambda + \gamma (c-j) )p_{j,0} =  \gamma (c-j+1) p_{j-1,0} + \mu p_{j+1,1} \\ \\
      1\leq n \leq j-1 \quad (\lambda + (c-j)\gamma + n \mu)p_{j,n}= \lambda p_{j,n-1} + (c-j+1)\gamma p_{j-1,n} + (n+1)\mu p_{j+1,n+1}\\ \\
      n=j \quad (\lambda + j\mu + (c-j)\gamma)p_{j,n} = \lambda p_{j,n-1} + (c-j+1)\gamma p_{j-1,n} + (j\mu + \xi)p_{j,n+1} (j+1)\mu p_{j+1,n+1}\\ \\
      n > j \quad (\lambda + j\mu + (c-j)\gamma + (n-j)\xi)p_{j,n} = \lambda p_{j,n-1} + (c-j+1)\gamma p_{j-1,n} + \left[ j \mu + (n+1-j)\xi \right]p_{j,n+1}
                 \end{array}        
\end{equation}

%
%
%
%
%

\begin{equation}
\label{eq-balancecc}
 J=c \left\{  \begin{array}{llll}
              n = 0 \quad \lambda p_{c,0} =  \gamma p_{c-1,0} \\
              1 \leq n \leq c-1 \quad (\lambda + n \mu)p_{c,n}= \lambda p_{c,n-1} + \gamma p_{c-1,n}\\
             n=c \quad (\lambda + c\mu)p_{c,n} = \lambda p_{c,n-1} + \gamma p_{c-1,n} + (c\mu + \xi)p_{c, n+1}\\
             n > c \quad (\lambda + c\mu + (n-c)\xi)p_{c,n} = \lambda p_{c, n-1} + \gamma p_{c-1,n} + \left[ c\mu + (n+1-c)\xi \right]p_{c,n+1}
                 \end{array}
        \right.          
\end{equation}

We define for every $0 \leq j \leq c$ the partial generating function $P_{j}(.)$

\begin{equation*}
P_{j}(z) = \sum_{n=0}^{\infty} z^{n} p_{j,n}
\end{equation*}

By multiplying by $z^n$ and summing over $n$, we obtain

\n\\
from (\ref{eq-balance0c}), for $j=0$

\begin{equation}
\label{equa diff0c}
\xi(1-z)P^{'}_{0}(z) = \left[ \lambda(1-z) + c\gamma \right]P_{0}(z) - \mu p_{1,1},
\end{equation}

\n
from (\ref{eq-balancejc}), for $1 \leq j \leq c-1$

\begin{align}
\label{equa diffjc}
\begin{split}
&\left[ ( \lambda z - \mu j)(1-z) + (c-j)\gamma z + j\xi (1-z) \right]P_{j}(z) - \xi z (1-z)P^{'}_{j}(z)\\
 &= (c-j+1)\gamma z P_{j-1}(z) + j \mu (1-z)\sum_{n=0}^{j}z^{n}p_{j,n} - \mu z \sum_{n=1}^{j} n z^{n}p_{j,n} \\
 & \qquad + \mu \sum_{n=1}^{j+1} n z^{n} p_{j+1,n} +j\xi (1-z) \sum_{n=0}^{j} z^{n}p_{j,n},
\end{split}
\end{align}

\n
from (\ref{eq-balancecc}), for $j=c$

\begin{align}
\label{equa diffcc}
\begin{split}
&\left[ ( \lambda z - \mu j)(1-z) + c\xi (1-z) \right]P_{c}(z) - \xi z (1-z)P^{'}_{c}(z)\\
 &= \gamma z P_{c-1}(z) - \mu z \sum_{n=1}^{c}nz^{n}p_{c,n} + c(1-z)(\xi - \mu)\sum_{n=0}^{c}z^{n}p_{c,n} - \xi (1-z)\sum_{n=1}^{c} n z^{n} p_{c,n}.
\end{split}
\end{align}

By setting $z=1$ in Equation (\ref{eq-balance0c}), (\ref{eq-balancejc}) and (\ref{eq-balancecc}) we get 

\n
for $j=0$, $P_{0}(1) = \frac{\mu}{c \gamma} p_{1,1}$,

\n
for $1 \leq j \leq c-1$: $(c-j+1)\gamma P_{j}(1) = (c-j+1)\gamma P_{j-1}(1) - \mu \sum _{n=1}^{j} n p_{j,n} + \mu \sum _{n=1}^{j+1} n p_{j+1,n}$ and $\gamma P_{c-1}(1) =  \mu \sum _{n=1}^{c} n p_{c,n}$, and by repeated substitution, we get
\begin{equation}
\label{P_{jc}(1)}
(c-j+1)\gamma P_{j}(1) =  \mu \sum _{n=1}^{j+1} n p_{j+1,n}
\end{equation}

\n\\
The differential equation (\ref{equa diff0c}) is similar to the differential equation (\ref{eq diff0}), so

\begin{equation}
\label{P0c}
P_{0}(z)  = p_{0,0} e^{\frac{\lambda}{\xi}z} (1-z)^{-\frac{c\gamma}{\xi}} \left[ 1 - \dfrac{A_{c}(z)}{A_{c}} \right],
\end{equation}
where $A_{c}(z) = \int_{0}^{z} e^{-\frac{\lambda}{\xi}s} (1-s)^{\frac{c\gamma}{\xi} - 1} ds,$ $A_{c} = A_{c}(1)$,
\begin{equation}
\label{p11c}
p_{1,1} = \frac{\xi}{\mu A_{c}} p_{0,0},
\end{equation}

\begin{equation}
\label{P_{0c}(1)}
P_{0}(1) = \frac{\xi}{c \gamma A_{c}} p_{0,0}.
\end{equation}
Once $p_{0,0}$ is calculated the partial generating function $P_{0}(z)$ is completely determined. 

\n
Dividing (\ref{equa diffjc}) by $-\xi z (1-z)$ yields

\begin{align}
\label{equa diffjc bis}
\begin{split}
&P'_{j}(z) - \left[ \frac{\lambda}{\xi} - \frac{1}{z}\left( \frac{\mu j}{\xi} - j\right) + \frac{(c-j)\gamma}{\xi}\frac{1}{1-z} \right]P_{j}(z) \\
 &= -\frac{(c-j+1)}{\xi (1-z)}\gamma P_{j-1}(z) - \frac{j \mu}{\xi z}\sum_{n=0}^{j}z^{n}p_{j,n} + \frac{\mu}{\xi (1-z)} \sum_{n=1}^{j} n z^{n}p_{j,n} \\
 & \qquad - \frac{\mu}{\xi z (1-z)}\sum_{n=1}^{j+1} n z^{n} p_{j+1,n} - \frac{j}{z} \sum_{n=0}^{j} z^{n}p_{j,n},
\end{split}
\end{align}

\n
Dividing (\ref{equa diffcc}) by $-\xi z(1-z)$ yields

\begin{align}
\label{equa diffcc bis}
\begin{split}
& P^{'}_{c}(z) - \left[ ( \frac{\lambda}{\xi} - \frac{1}{z}\left( \frac{\mu j}{\xi} - c\right) \right]P_{c}(z) \\
 &= - \frac{\gamma}{\xi (1-z)} P_{c-1}(z) + \frac{\mu}{\xi (1-z)} \sum_{n=1}^{c}nz^{n}p_{c,n} - \frac{c(\xi - \mu)}{\xi z}\sum_{n=0}^{c}z^{n}p_{c,n} + \frac{1}{z}\sum_{n=1}^{c} n z^{n} p_{c,n}.
\end{split}
\end{align}

\n
The set of equations (\ref{equa diffjc bis}) and (\ref{equa diffcc bis}) allow us to solve recursively  the differential equations using similar technic as in previous section. Once the unknown probabilities $\{ p_{j,n} \}$, for $j=1,\ldots , c$, $n =0, \ldots ,j$ and $P_{j}(1)$, for $j=0,\ldots , c$ all the partial differential equations are completely determinded.

Multiplying both sides of (\ref{equa diffjc bis}) by $e^{-\frac{\lambda}{\xi}z} z^{(\frac{\mu j}{\xi} - j)(1-z)^{(c-j)\gamma / \xi}}$ and integrating we get

\begin{align*}
& e^{-\frac{\lambda}{\xi}z} z^{(\frac{\mu j}{\xi} - j)(1-z)^{(c-j)\gamma / \xi}} P_{j}(z)\\
&= -\frac{(c-j+1)\gamma}{\xi}\int_{0}^{z} e^{-\frac{\lambda}{\xi}s} s^{(\frac{\mu j}{\xi} - j)}(1-s)^{(c-j)\gamma / \xi - 1}P_{j-1}(s)ds\\
&= -j\left( 1+\frac{\mu}{\xi} \right)\sum_{n=0}^{j}p_{j,n} \int_{0}^{z} e^{-\frac{\lambda}{\xi}s} s^{(\frac{\mu j}{\xi} +n - 2)}(1-s)^{(c-j)\gamma / \xi }ds\\
&=\frac{\mu}{\xi}\sum_{n=1}^{j} n p_{j,n}\int_{0}^{z} e^{-\frac{\lambda}{\xi}s} s^{(\frac{\mu j}{\xi} +n - 1)}(1-s)^{(c-j)\gamma / \xi -1 }ds\\
&=-\frac{\mu}{\xi}\sum_{n=1}^{j+1} n p_{j+1,n}\int_{0}^{z} e^{-\frac{\lambda}{\xi}s} s^{(\frac{\mu j}{\xi} +n - 2)}(1-s)^{(c-j)\gamma / \xi -1 }ds.
\end{align*}

Thus,

\begin{equation}
\label{Pjc}
\begin{split}
 P_{j}(z) &= e^{\frac{\lambda}{\xi}z} z^{-(\frac{\mu j}{\xi} - 1)}(1-z)^{-\frac{(c-j)\gamma}{\xi}} \{
 -\frac{(c-j+1)\gamma}{\xi} k_{1}(z) -j\left( 1+\frac{\mu}{\xi} \right) \sum_{n=0}^{j}p_{j,n}k_{2}(z)  + \frac{\mu}{\xi}\sum_{n=1}^{j} n p_{j,n} k_{3}(z)\\ 
 &\qquad -\frac{\mu}{\xi}\sum_{n=1}^{j+1} n p_{j+1,n} k_{4}(z)\}.
 \end{split}
\end{equation}

Solve Equation (\ref{equa diffcc bis}) yields by multiplying by $e^{-\frac{\lambda}{\xi}z} z^{(\frac{\mu j}{\xi} - c)}$

\begin{equation*}
\begin{split}
P_{c}(z)&= e^{\frac{\lambda}{\xi}z} z^{-(\frac{\mu j}{\xi} - c)} \{ -\frac{\gamma}{\xi} \int_{0}^{z}e^{-\frac{\lambda}{\xi}s} s^{(\frac{\mu j}{\xi} - c +n)}(1-s)^{-1} P_{c-1}(s)ds + \frac{\mu}{\xi}\sum_{n=1}^{c} n p_{c,n} \int_{0}^{z}e^{-\frac{\lambda}{\xi}s} s^{(\frac{\mu j}{\xi} - c +n)}(1-s)^{-1}ds \\
&\qquad -  \frac{c(\xi - \mu)}{\xi} \sum_{n=0}^{c} p_{c,n} \int_{0}^{z}e^{-\frac{\lambda}{\xi}s} s^{(\frac{\mu j}{\xi} - c +n -1)}ds + \sum_{n=1}^{c}n p_{c,n}\int_{0}^{z}e^{-\frac{\lambda}{\xi}s} s^{(\frac{\mu j}{\xi} - c +n -1)}ds\}.
\end{split}
\end{equation*}

Thus

\begin{equation}
\label{Pcc}
\begin{split}
P_{c}(z)= e^{\frac{\lambda}{\xi}z} z^{-(\frac{\mu j}{\xi} - c)} \left\lbrace -\frac{\gamma}{\xi} \ell_{1}(z) + \frac{\mu}{\xi}\sum_{n=1}^{c} n p_{c,n} \ell_{2}(z) - \frac{c(\xi - \mu)}{\xi} \sum_{n=0}^{c} p_{c,n} \ell_{3}(z)+ \sum_{n=1}^{c}n p_{c,n}\ell_{3}(z)\right\rbrace.
\end{split}
\end{equation}

We get for $z=1$

\begin{equation}
\label{P_{cc}(1)}
P_{c}(1)= e^{\frac{\lambda}{\xi}} \left\lbrace -\frac{\gamma}{\xi} \ell_{1}^{c-1}(1) + \frac{\mu}{\xi}\sum_{n=1}^{c} n p_{c,n} \ell_{2}(1) - \frac{c(\xi - \mu)}{\xi} \sum_{n=0}^{c} p_{c,n} \ell_{3}(1)+ \sum_{n=1}^{c}n p_{c,n}\ell_{3}(1)\right\rbrace.
\end{equation}

Equations (\ref{P0c}), (\ref{Pjc}) and (\ref{Pcc}) express the partial PGF $P_{j}(z)$, $j=0, \ldots ,c$ in terms of $(c+1)(c+2)/2$ unknown probabilities $\{ p_{j,n} \}_{\substack{0 \leq j \leq c\\0\leq n \leq j}}$. Those probabilities are required in order to completely determine $P_{j}(z)$, $j=0, \ldots , c$. The primary idea is to express the $c(c+1)/2$ probabilities $\{ p_{j,n} \}_{\substack{1 \leq j \leq c\\0\leq n \leq j-1}}$ in terms of those $\{ p_{j,j} \}_{1 \leq j \leq c-1}$. The $c(c+1)/2$ equations connecting $\{ p_{j,n} \}_{\substack{0 \leq j \leq c\\0\leq n \leq c}}$ (except $p_{c,c}$) with $\{ p_{j,j} \}_{0 \leq j \leq c-1}$ are taken from the balance equations for $j=1, \ldots ,c$, $n=0, \ldots , j-1$. The Equation (\ref{p11c}) expresses $p_{1,1}$ in term of $p_{0,0}$, this implies that the probabilities $\{ p_{j,n} \}_{\substack{1 \leq j \leq c\\0\leq n \leq c}}$ (except $p_{c,c}$) may be expressed in terms of $\{ p_{j,j} \}_{\substack{0 \leq j \leq c-1 \\ j \neq 1}}$. For example, if $c=3$ the $6$ balance equations for $j=1, \ldots , 3$, $n=0, \ldots , j-1$ together with (\ref{p11c}) connect $p_{1,0}, p_{1,1}, p_{2,0}, p_{2,1}, p_{3,0}, p_{3,1}, p_{3,2}$ in terms of $p_{0,0}$ and/or  $p_{2,2}$.

The second idea consists in expressing the $\{ p_{j,j} \}_{2 \leq j \leq c}$ in term of $p_{0,0}$, so we have to find $c-1$ equations. Taking the limit in (\ref{Pjc}) yields for $j=1, \ldots , c-1$

\begin{equation}
\begin{split}
0 = -\frac{(c-j+1)\gamma}{\xi} k_{1}^{j-1}(1) -j\left( 1+\frac{\mu}{\xi} \right) \sum_{n=0}^{j}p_{j,n}k_{2}(1)  + \frac{\mu}{\xi}\sum_{n=1}^{j} n p_{j,n} k_{3}(1) -\frac{\mu}{\xi}\sum_{n=1}^{j+1} n p_{j+1,n} k_{4}(1)
 \end{split}
\end{equation}

since $0 < P_{j}(1) = \sum_{n=0}^{\infty} p_{j,n} < \infty$ and $\lim_{z \rightarrow 1}(1-z)^{- \frac{(c-j)\gamma}{\xi}} = + \infty$. These $c-1$ above equations gives the probabilities $p_{j,j}$, $j=2, \ldots , c$ in term of $p_{0,0}$.

Finally the norming equation $\sum_{j=0}^{c}P_{j}(1) = 1$ combining with (\ref{P_{jc}(1)}) and (\ref{P_{cc}(1)}) yields $p_{0,0}$ and thus all the unknown probabilities.

\bibliographystyle{plainnat}
\bibliography{biblio-markovian-queue}

\end{document}